\documentclass[reqno]{amsart}
\usepackage{amsmath,amsfonts,amsgen,amstext,amsbsy,amsopn,amsthm,
  amssymb}

\newtheorem{Lemma}{Lemma} \newtheorem{Theorem}{Theorem}
\newtheorem{proposition}{Proposition}
 \newtheorem{corollary}{Corollary}
\theoremstyle{definition} 
\newtheorem{Remark}{Remark} 

\numberwithin{equation}{section} \newcommand{\Ss}{\mathbb{S}}
\newcommand{\R}{\mathbb{R}}

 \newcommand{\W}{\mathcal{W}}
\newcommand{\B}{\mathcal{B}}

\newcommand\eps\varepsilon \newcommand\V{\mathcal{V}}

\DeclareMathOperator{\sgn}{sgn}
 \DeclareMathOperator{\const}{const}

\begin{document}
\title[Operators with degenerate symbol] {Asymptotic behavior of
  eigenvalues of Schr\"odinger type operators with degenerate kinetic
  energy}

\author{Christian Hainzl} \address{Departments of Mathematics and
  Physics, UAB, 
  Birmingham AL 35294, USA} \email{hainzl@math.uab.edu}

\author{Robert Seiringer} \address{Department of Physics, Princeton
  University, Princeton NJ 08542-0708, USA}
\email{rseiring@princeton.edu}

\date{August 26, 2008.\\ \copyright\,2008 by the authors.  This paper may be
  reproduced, in its entirety, for non-commercial purposes.  }

\begin{abstract}
  We study the eigenvalues of Schr\"odinger type operators $T +
  \lambda V$ and their asymptotic behavior in the small coupling limit
  $\lambda \to 0$, in the case where the symbol of the kinetic energy,
  $T(p)$, strongly degenerates on a non-trivial manifold of
  codimension one.
\end{abstract}

\maketitle

\section{Introduction}

In several recent papers attention has been drawn  to
Schr\"odinger type operators on $L^2(\R^n)$ of the form
\begin{equation}\label{hl}
  H_\lambda = T(i\nabla) + \lambda V(x)\,, 
\end{equation}
where the non-negative symbol $T(p)$ degenerates on a manifold $S$
of codimension one, $V(x)$ is a real-valued potential, and $\lambda
> 0$ denoting the coupling parameter. The degeneracy of $T$ causes a
high instability of the lower edge of the spectrum of $H_\lambda$ and
gives rise to spectral properties which are comparable to the case of
Schr\"odinger operators in one dimension.  Operators of the type
(\ref{hl}) have appeared in the study of the roton spectrum of liquid
helium II \cite{KC}, matrix Hamiltonians in spintronics
\cite{BGP,BR,CM}, as well in the elasticity theory \cite{F,FW}.

Typically, we think of $T(p)$ as originating from a smooth symbol,
$P(p)$, which vanishes on $S$ and has no critical points in the
neighborhood of $S$, with
\begin{equation}\label{defr}
  T(p) = |P(p)|^r 
\end{equation}
for some parameter $1 \leq r < \infty$.  As pointed out by Laptev,
Safronov and Weidl in \cite{LSW}, due to the singularity of the
resolvent of $T$ on $S$ the spectrum of $T + \lambda V$ is mainly
determined by the behavior of the potential $V$ close to $S$.  More
precisely, an important role is played by an operator acting on
functions on $S$, i.e., $\V_S: L^2(S) \to L^2(S)$, given by
\begin{equation}\label{VS}
  (\V_S u)(p) = \frac 1{\sqrt{|\nabla
      P(p)|}(2\pi)^{n/2}}\int_S \hat V(p-q) \frac{u(q)}{\sqrt{|\nabla
      P(q)|}} dq\,,
\end{equation}
with $dq$ being the Lebesgue measure on $S$ and $\hat V(p) =
(2\pi)^{-n/2}\int_{\R^n} e^{-ix\cdot p} V(x) dx$ denoting the Fourier
transform of $V(x)$. In particular, it was shown in \cite{LSW} that
$T+ \lambda V$ has infinitely many negative eigenvalues if $V$ is
negative.

Operators of the type (\ref{VS}) appeared already earlier in \cite{BY}
in the study of scattering phases. They play a crucial role in the
study of the non-linear Bardeen-Cooper-Schrieffer (BCS) gap equation
of superfluidity \cite{BCS,Leg}. In fact, it was shown in
\cite{FHNS,HS,HS2} that the lowest eigenvalue of $\V_S$ is related to
the critical temperature for the existence of solutions of the BCS gap
equation. In this case, $T(p)$ is roughly of the form $|p^2-\mu|$ for
$\mu>0$, $p\in \R^3$, and hence $S$ is the two dimensional sphere of
radius $\sqrt\mu$.

The goal of the present paper is to generalize the results and
techniques of \cite{FHNS,HS} to a large class of manifolds $S$ and
kinetic symbols $T(p)$. We shall show that corresponding to any
negative eigenvalue, $a^i_S$, of the compact operator $\V_S$ there
exists a negative eigenvalue, $- e_i(\lambda)$, of $T + \lambda
V$. Moreover, in Theorem~\ref{Thm1gen} we study the asymptotic
behavior of $e_i(\lambda)$ as $\lambda\to 0$ and show that
\begin{equation}\label{asymptbeh}
  \lim_{\lambda \to 0} \lambda f(e_i(\lambda)) =
  -1/a^i_S\,, 
\end{equation}
where the function $f$ depends on the value of $r$ in (\ref{defr}) as
\begin{equation}\label{fe}
  f(e) = \left\{ \begin{aligned}
      & \frac{2\pi}{r \sin(\pi/r)} \frac 1{ e^{(r-1)/r} }    
\, \, &{\rm if}& \,\,\, r > 1\\
      & 2 \ln [1+1/e]   \,\, &{\rm if}& \,\,\, r=1.
    \end{aligned}
  \right.
\end{equation}
We shall also relate the eigenvector $\psi_\lambda^i$ of $H_\lambda$
corresponding to the eigenvalue $-e_i(\lambda)$ to the eigenvector
$u_i$ of $\V_S$ with eigenvalue $a_S^i$.  We shall find that after
appropriate normalization $\psi_\lambda^i$ converges to 
\begin{equation}
  \int_{S}
  e^{ix\cdot p} \frac{u_i(p)}{\sqrt{(2\pi)^{n}|\nabla P(p)|}} dp
\end{equation}
in the limit $\lambda \to 0$ in a suitable sense.

If $1\leq r < 2$ our methods enable us to find the next to leading
order term of $\lambda f(e_i(\lambda))$ as $\lambda\to 0$. This is
the content of Theorem \ref{thm2statement}.

\section{Main results}

We consider operators on $L^2(\R^n)$, $n\geq 2$,  of the form
\begin{equation}
  H_\lambda = T(i\nabla) + \lambda V(x)\,.
\end{equation}
The symbol of the kinetic operator, $T(p)$, attains its minimum on a
manifold of codimension one. For convenience let us assume that the
minimum value is zero, and let
\begin{equation}\label{defs}
  S = \left\{p\in \R^n\,|\, T(p) =0\right\}\,.
\end{equation}
It is not being assumed that $S$ is connected, but it should consist of only finitely many connected components. We shall further assume that there exists a $\sigma>0$ and a compact
neighborhood $\Omega\subset \R^n$ of $S$ containing $S$, with the
property that the distance of any point in $S$ to the complement of
$\Omega$ is at least $\sigma$. Moreover, we assume that
\begin{itemize}
\item[(i)] $T(p) = |P(p)|^r$ for some locally bounded, measurable
  function $P$, with $1\leq r < \infty$, and $P\in C^2(\Omega)$,
\item[(ii)] $|\nabla P|$ does not vanish in $\Omega$,
\item[(iii)] for some constants $C_1>0$, $C_2>0$ and $s>0$, $T \geq
  C_1 |p|^{s} + C_2$ for $p\not\in \Omega$.
\end{itemize}
These assumptions appear naturally in all recent applications
mentioned in the introduction. They could be relaxed in various ways,
but we shall not try to do so in order to avoid unilluminating
complications in the proofs.

Since $S$ in (\ref{defs}) is the zero set of the function $P\in
C^2(\Omega)$, and $\nabla P \neq 0$ in $\Omega$ by assumption, we
conclude that $S$ is a nice submanifold of codimension one. In
particular, if $V\in L^1(\R^n)$ then $\hat V(p)$ is a bounded,
continuous function and hence (\ref{VS}) defines a compact (in fact,
trace-class) operator $V_S$ on $L^2(S)$.

In the following, it will be useful to introduce the operator $F_S:
L^1(\R^n) \to L^2(S)$, which is obtained by restricting the Fourier
transform to $S$ and multiplying by $|\nabla P|^{-1/2}$, i.e.,
\begin{equation}
  \left(F_S\varphi\right)(p) = \frac {1}{\sqrt{(2\pi)^n}\sqrt{|\nabla P
      (p)|}} \int_{\R^n} e^{-ix\cdot p} \varphi(x) dx\Big|_{p\in S}
  \,.  
\end{equation}
Its adjoint, $F_S^* : L^2(S) \to L^\infty(\R^n)$, is given by
\begin{equation}
  \left (F^*_S u\right)(x) = \frac {1}{\sqrt{(2\pi)^n}} \int_{S}
  \frac{e^{ix\cdot p} }{\sqrt{|\nabla P(p)|}}u(p) dp\, .  
\end{equation}
Then $\V_S$ is (\ref{VS}) equals $F_S V F_S^*$. Note that $V^{1/2}
F_S^*$ is a bounded operator if $V\in L^1(\R^n)$.

For $i=1,2,\dots$, let $a^i_S < 0$ be the negative eigenvalues,
counting multiplicity, of $\V_S$, and let $u_i$ be its eigenvectors,
i.e.,
\begin{equation}
  \V_S u_i = a^i_S u_i\,, \qquad u_i \in L^2(S)\,.
\end{equation}
The following theorem shows that it is possible to associate to any
such $a^i_S$ a negative eigenvalue $- e^i(\lambda)$ for
$H_\lambda$. Moreover, we will recover the asymptotic behavior of
$e_i(\lambda)$ in the limit $\lambda \to 0$. A similar statement can
be made about the corresponding eigenvectors. The theorem is a
generalization of \cite[Theorem~1]{FHNS}.

\begin{Theorem} \label{Thm1gen} Let $T(p)$ satisfy the assumptions
  above, and let $V\in L^1(\R^n) \cap L^{n/s}(\R^n)$ if $n>s$, $V\in
  L^1(\R^n) \cap L^{1+\eps}(\R^n)$ for some $\eps>0$ if $n=s$, and
  $V\in L^1(\R^n)$ if $n<s$. Additionally we assume that $\iint |V(x)
  ||x - y|^\kappa |V(y)|dxdy < \infty$, with $\kappa = 2$ if $T$ is
  not a radial function, and $\kappa=1$ if $T$ is radial and
  $n=2$. Then
  \begin{itemize}
  \item[(i)] for every negative eigenvalue $a_S^i<0$ of $\V_S$,
    counting multiplicity, and every $\lambda>0$, there is a negative
    eigenvalue $-e_i(\lambda)<0$ of $H_\lambda = T + \lambda V$ such
    that
    \begin{equation}\label{limf2.2}
      \lim_{\lambda \to 0} \lambda f(e_i(\lambda)) = -1/a_S^i \,.
    \end{equation}
    The function $f$ is defined in (\ref{fe}).
  \item[(ii)] for every eigenvector $\psi^i_\lambda \in L^2(\R^n)$ of
    $H_\lambda$, corresponding to the eigenvalue $-e_i(\lambda)$, there is an
    eigenvector $u_i \in L^2(S) $ of $\V_S$ corresponding to $a_S^i$
    such that after appropriate normalization
    \begin{equation}\label{conveigvhlambdatoeigenvvs}
      V^{1/2} \psi^i_\lambda \to V^{1/2} F_S^* u_i \quad 
     \text{as $\lambda \to 0$, strongly in $L^2(\R^n)$}\,.
    \end{equation}
  \item[(iii)] if $r < 2$ all other possible eigenvalues
    $-e_j(\lambda)$ of $H_\lambda$ satisfy $ f(e_j(\lambda)) \geq c
    \lambda^{-2}$ for some constant $c>0$.
  \item[(iv)] if $r < 2$ and $\V_S \geq 0$, and there exists an
    $\delta$ such that also $ F_S(V - \delta |V|)F^*_S \geq 0$
    then $H_\lambda \geq 0$ for $\lambda$ small enough.
  \end{itemize}
\end{Theorem}

Equation (\ref{limf2.2}) implies, in particular, that
\begin{equation}
  e_i(\lambda) = \left \{
    \begin{aligned}
      & \left(\frac{2\pi}{r \sin(\pi/r)} \lambda |a^i_S|\right)^{r/(r-1)}
\left(1 +  o(1)\right) 
\, \, &{\rm if}& \,\,\, r > 1\\
      &\exp\left(-\frac 1{2\lambda |a^i_S|}\left(1 +
          o(1)\right)\right) \,\, &{\rm if}& \,\,\, r=1
    \end{aligned}
  \right. 
\end{equation}
as $\lambda\to 0$.  On the other hand (iii) guarantees that all
possible eigenvalues of $H_\lambda$ not corresponding to a negative
eigenvalue of $\V_S$ satisfy
\begin{equation}
  e_j(\lambda) \leq  \left \{
    \begin{aligned}
      &\const \lambda^{2r/(r-1)} \, \, &{\rm if}& \,\,\, 1 <r< 2\\
      & \exp\left(-\const \lambda^{-2}\right) \,\, &{\rm if}& \,\,\,
      r=1 \,.
    \end{aligned}
  \right.
\end{equation}
The following immediate corollary of Theorem~\ref{Thm1gen} generalizes results in \cite{LSW,P}.

\begin{corollary}\label{Cor}
  Let the assumptions be as in Theorem \ref{Thm1gen}.
  \begin{itemize}
  \item[(i)] Then, for all $\lambda > 0$, the operator $H_\lambda$ has
    at least as many negative eigenvalues as $\V_S$ does.
  \item[(ii)] If $V(x) \leq 0$ and does not vanish a.e., then $\V_S$
    (and consequently $H_\lambda$) has infinitely many negative
    eigenvalues.
  \end{itemize}
\end{corollary}
\begin{proof}
  The negative eigenvalues, $-e_i(\lambda)/\lambda$, of the operator
  $H_\lambda/\lambda = T/\lambda + V$ are monotonically decreasing in
  $\lambda$ since $T\geq 0$. Consequently if $-e_i(\bar \lambda) < 0$
  then $-e_i(\lambda)$ is necessarily negative for all $\lambda \geq
  \bar \lambda$. Thus $(i)$ follows immediately from Theorem
  \ref{Thm1gen} (i).

  If $V \leq 0$, then $\V_S\leq 0$ and all eigenvalues of $\V_S$ are
  necessarily non-positive. We shall argue that $0$ cannot be an
  eigenvalue of $\V_S$ since for any non-zero function $\varphi\in
  L^2(S)$, $F_S^* \varphi$ can vanish at most on a subset of $\R^n$ of
  codimension one. This follows from the fact that
  $(F_S^*\varphi)(x_1,\dots,x_n)$ is analytic in each component $x_i$,
  and therefore can only have isolated zeros in each component.
  Consequently $(\varphi, \V_S \varphi) = \int_{\R^n} |(F_S^* u)(x)|^2
  V(x)dx < 0$ for any $u$. This implies (ii).
\end{proof}

\begin{Remark}
 In the BCS gap equation of superfluidity at zero temperature
  \cite{HHSS,FHNS,HS,HS2} the kinetic energy operator $T(p) = |p^2 -
  \mu|$ appears, with $\mu>0$ being the chemical potential. In this
  case $r=1$ and hence $f(e)= 2\ln(1/e)$. Therefore, the eigenvalues
  of $T + \lambda V$ are exponentially small and satisfy $e_i(\lambda)
  \sim e^{-\frac 1{2 \lambda |a^i_S|}}$.
\end{Remark}

\begin{Remark} In the study of the roton spectrum in liquid helium \cite{Landau} a
  kinetic energy of the type $T(p) = \frac{(|p| - p_0)^2}{2\mu} +
  \Delta$ arises, with $p_0,\mu, \Delta > 0$. In this case Theorem
  \ref{Thm1gen} implies that the eigenvalues depend quadratically on
  $\lambda$ for small $\lambda$, i.e., $e_i(\lambda) - \Delta \sim
  (\lambda |a_S^i|)^2$, similar to the case of Schr\"odinger operators
  in one dimension \cite{simon}.
\end{Remark}

\begin{Remark} The convergence property \eqref{conveigvhlambdatoeigenvvs} can be
  particularly useful in the case where the manifold $S$ is a sphere
  and the potential $V$ is radial, since the eigenfunctions of $\V_S$
  are known explicitly. In the case $n=3$, for instance, they are the
  spherical harmonics. If additionally $\hat V \leq 0$ then the
  constant function on $S$ is the ground state of $\V_S$. This
  property was important in \cite{HS} where a precise characterization
  of the asymptotic behavior of the solution of the BCS gap
    equation of superfluidity was given.
\end{Remark}

\begin{Remark}
  In the case of trapped modes for an elastic plate in \cite{F} a
  small coupling asymptotics was derived in the case where $S$ is a
  circle in $\R^2$.
\end{Remark}

In the following let $r < 2$. In this case, we shall now state a more
precise characterization of the asymptotic behavior of the eigenvalues
of $H_\lambda$ as $\lambda\to 0$. More precisely, we will recover the
next order in $\lambda$.

It will be shown in Lemma~\ref{lem2} that the quadratic form
\begin{equation}\label{defws}
  \left(u, \W_S u\right) = \lim_{e \to 0} \left( u, F_S V \left(\frac 1{T + e}  
- f(e) F^*_S F_S\right)  V F^*_S u\right) 
\end{equation}
defines a bounded operator on $L^2(S)$.  For $\lambda > 0$ let further
\begin{equation}\label{defB}
  \B_S = \V_S - \lambda \W_S
\end{equation}
and let $b_S^i(\lambda) < 0$ denote the negative eigenvalues of
$\B_S$. The following theorem is a generalization of
\cite[Theorem~1]{HS}.
 
\begin{Theorem}\label{thm2statement} Let $T$ and $V$ be as in
  Theorem~\ref{Thm1gen} and assume that $r < 2$. Then
  \begin{itemize}
  \item[(i)] If $\lim_{\lambda \to 0} b_S^i(\lambda) < 0$ then
    $H_\lambda= T+\lambda V$ has, for small $\lambda$, a corresponding
    negative eigenvalue $-e_i(\lambda) < 0$, with
    \begin{equation}
      \label{secondorderofei} \lim_{\lambda \to 0} \left [f(e_i(\lambda))
        + \frac 1{\lambda b_S^i(\lambda)} \right] = 0.
    \end{equation}
  \item[(ii)] If the kernel of $\V_S$ is not empty then there exists
    at least one corresponding negative eigenvalue of $H_\lambda$.
  \end{itemize}
\end{Theorem}

\begin{Remark}
  If $a^i_S < 0$ is a non-degenerate eigenvalue of $\V_S$ and $u_i$ is
  the corresponding eigenvector, then first order perturbation theory
  implies that the corresponding eigenvalue of $\B_S$ satisfies
  \begin{equation}
    b_S^i(\lambda) =  a_S^i - \lambda  (u_i, \W_S u_i) + o(\lambda)\,.
  \end{equation}
  Hence (\ref{secondorderofei}) can be rewritten in the form
$$
\lim_{\lambda \to 0} \left [f(e_i(\lambda)) + \frac 1{\lambda a^i_S} +
  \frac{(u_i, \W_S u_i)}{(a_S^i)^2}\right] = 0.
$$
A similar expression holds in case $a_i<0$ is $k$-fold degenerate,
with $(u_i,\W_Su_i)$ replaced by the eigenvalues of the $k\times k$
matrix $(u_i^{(j)},\W_S u_i^{(l)})$, where $u_i^{(j)}$ denotes the
eigenvectors of $\V_S$ corresponding to the eigenvalue $a^i_S$.
\end{Remark}

\section{Proofs}

According to the Birman-Schwinger principle, the operator $H_\lambda$
has a negative eigenvalue $-e<0$ if and only if the compact operator
\begin{equation}\label{BS-operator}
  \lambda V^{1/2}\frac 1{T+ e} |V|^{1/2}
\end{equation}
has an eigenvalue $-1$. Here, we use the usual convention $V^{1/2} =
\sgn(V)|V|^{1/2}$. Note that $V^{1/2}(T+ e)^{-1} |V|^{1/2} $ is
actually a Hilbert-Schmidt operator for $e>0$. This follows from the
Hardy-Littlewood-Sobolev inequality \cite[Theorem~4.3]{LL} and our
assumptions on $T$ and $V$.

More precisely, if
\begin{equation} \label{eigenvalue-equ-e} H_\lambda \psi_\lambda = - e
  \psi_\lambda
\end{equation}
for $\psi_\lambda \in L^2(\R^n)$ and $e>0$, then
\begin{equation}\label{thm1birmschwinger}
  \lambda V^{1/2}\frac 1{T+ e} |V|^{1/2} \phi_\lambda = -\phi_\lambda\,,
\end{equation}
where $\phi_\lambda = V^{1/2} \psi_\lambda$. It is, in fact, not
difficult to see $\phi_\lambda\in L^2$, since $|V|$ is infinitesimally
form-bounded with respect to $T$ under our assumptions on $T$ and
$V$. On the other hand \eqref{thm1birmschwinger} implies
\eqref{eigenvalue-equ-e} by choosing
\begin{equation}\label{eigenvbirmanschwinger}
  \psi_\lambda = \frac 1{T+e} |V|^{1/2} \phi_\lambda
\end{equation}
which is in $L^2(\R^3)$ since $T\geq 0$, $e>0$ and the operator
$|V|^{1/2}(T+e)^{-1}|V|^{1/2}$ is bounded.

Our results will rely on the fact that the singular part of the
\eqref{BS-operator} as $e\to 0$ is governed by the operator $V^{1/2}
F_S^* F_S |V|^{1/2}$, which is isospectral to $\V_S = F_S V F_S^*$.

In the following, let $M_e$ denote the bounded operator
\begin{equation}\label{defMe}
  M_e = V^{1/2} \left(\frac 1{T + e}  - f(e) F^*_S F_S\right) |V|^{1/2}\,.
\end{equation}

\begin{proposition}\label{prop1}
  Assume that $1 + \lambda M_e$ is
  invertible. Then $H_\lambda$ has an eigenvalue $-e < 0 $ if and only
  if the selfadjoint operator
  \begin{equation}\label{volleoperatorvonSnachS}
    F_S|V|^{1/2} \frac {\lambda f(e)}{ 1 + \lambda  M_e } V^{1/2}F^*_S\, : \, L^2(S) \to L^2(S) 
  \end{equation}
  has an eigenvalue $-1$.  Furthermore, if $u
  \in L^2(S)$ is an eigenvector of \eqref{volleoperatorvonSnachS} with
  eigenvalue $-1$, then
  \begin{equation}\label{beziehungeigenvectorhlzuoperatoraufS}
    \psi_\lambda = \frac 1{T+e} |V|^{1/2} \frac 1{1 + \lambda M_e
    }V^{1/2} F^*_S u
  \end{equation}
  is an eigenvector of $H_\lambda$ in $L^2(\R^n)$ with eigenvalue
  $-e<0$.
\end{proposition}

\begin{proof}
  According to the Birman-Schwinger principle discussed above,
  $H_\lambda $ having an eigenvalue $-e<0 $ is equivalent to the fact
  that $\lambda V^{1/2}\frac 1{T+ e} |V|^{1/2} + 1$ has a zero
   eigenvalue.  Using the definition of $M_e$ in
  \eqref{defMe} this implies that
  \begin{multline}\label{equ:2thm1}
    \lambda V^{1/2}\frac 1{T+ e} |V|^{1/2} + 1=\lambda f(e) V^{1/2}
    F_S^* F_S |V|^{1/2} + \lambda M_e + 1 \\ = (1 + \lambda M_e)
    \left(\frac { \lambda f(e)}{1 + \lambda M_e }  V^{1/2} F^*_S F_S
      |V|^{1/2} + 1\right)
  \end{multline}
  has an eigenvalue $0$. Under the assumption that $1 + \lambda M_e$
  is invertible we conclude that
  \begin{equation}\label{po}
    \frac { \lambda f(e)}{1 + \lambda  M_e }
    V^{1/2} F^*_S F_S |V|^{1/2}
  \end{equation}
  must have $-1$ as an eigenvalue. The fact that (\ref{po}) is
  isospectral to \eqref{volleoperatorvonSnachS}, together with the
  observation that all the arguments work in either direction, implies
  the first part of the theorem.  The second part of the theorem is an
  easy consequence of \eqref{eigenvbirmanschwinger}.
\end{proof}

In order to apply Proposition \ref{prop1} we need a bound on the
operator $M_e$ in (\ref{defMe}). The bound we derive will be expressed
in terms of the function
\begin{equation}
  g(e) =  \left\{ \begin{aligned}
      & 1                 \,\, &{\rm if}& \,\,\, 1 \leq r < 2\\
      & 1 + \ln [1+ 1/e]   \,\, &{\rm if}& \,\,\, r=2 \\
      & 1 + e^{2-r}       \,\, &{\rm if}& \,\,\, r>2\,.
    \end{aligned}
  \right.
\end{equation}
The following lemma is the basis for our analysis.

\begin{Lemma}\label{lem1}
  Let
  \begin{equation}\label{defcala}
    \mathcal{A}(V) = \left\{
      \begin{aligned}
        & \|V\|_{n/s}         \,\, &{\rm if}& \,\,\, n > s\\
        & \|V\|_{1+\eps}  \,\, &{\rm if}& \,\,\, n = s\\
        & \|V\|_1 \,\, &{\rm if}& \,\,\, n < s\,.
      \end{aligned}\right.
  \end{equation}
  Then
  \begin{equation}\label{menonr}
    \| M_e \| \leq \const \left( g(e)\left[\|V\|_{1}+
        \left(\iint dxdy |V(x)||x - y|^\kappa |V(y)| \right)^{1/2}
      \right]+\mathcal{A}(V)\right)
  \end{equation}
  with $\kappa=2$.  If $T(p)$ is radial, then (\ref{menonr}) holds
  with $\kappa=0$ for $n\geq 3$ and $\kappa=1$ for $n=2$.
\end{Lemma}

Let us postpone the proof of this lemma until the end of the
section. The lemma says, in particular, that when $r<2$ the family of
operators $M_e$ is uniformly bounded. The limit of $M_e$ as $e\to 0$
actually exist in the operator norm topology. This is the content of
the next lemma, whose proof will also be given at the end this
section.

\begin{Lemma}\label{lem2}
  Assume that $r<2$. Then the limit
  \begin{equation}
    M_0 = \lim_{e\to 0} M_e 
  \end{equation}
  exists in the operator norm topology.
\end{Lemma}

An explicit expression of $M_0$ will be given in the proof of
Lemma~\ref{lem2}.  We note that the operator $\W_S$ in (\ref{defws})
equals $\W_S = F_S |V|^{1/2} M_0 V^{1/2} F_S^*$.

We have now all tools in hand to prove our main theorems.

\begin{proof}[Proof of Theorem \ref{Thm1gen}]
  By assumption, the operator $\V_S$ has negative eigenvalues $a_S^i$
  with corresponding eigenfunctions $u_i \in L^2(S)$. 
  We shall show that for every $a^i_S<0$ and $\lambda$ small enough
  there exists a function $e_i(\lambda)>0$, with $\lim_{\lambda\to 0}
  \lambda f(e_i(\lambda)) = -1/a^i_S$, such that the selfadjoint
  operator \eqref{volleoperatorvonSnachS} has an eigenvalue $-1$ for
  $e = e_i(\lambda)$. Because of Proposition~\ref{prop1} this implies
  $(i)$.

  For this purpose consider the selfadjoint operator
  \begin{equation}\label{Gl}
    G(\lambda,e)= |V|^{1/2}\left( \frac 1{ 1+\lambda M_e} - 1\right)V^{1/2}\,.
  \end{equation}
  In terms of $G(\lambda,e)$, the
  operator~(\ref{volleoperatorvonSnachS}) can be expressed as
  \begin{equation}\label{thmG}
    \lambda f(e) (\V_S + F_S G(\lambda,e)F^*_S)\,.
  \end{equation}
  Let us first consider first the case $r < 2$, where
  $g(e)=1$. According to Lemma~\ref{lem1}, $M_e$ is uniformly bounded
  and hence $1+\lambda M_e$ is invertible for small
  $\lambda$. Therefore,
  \begin{equation}\label{gb}
    \|F_SG(\lambda,e)F^*_S\| \leq \const \|V\|_{1} \frac{\lambda
      \|M_e\|}{1 - \lambda \|M_e\|}\,,
  \end{equation}
  where we used that $\|F_S|V|F_S^*\| \leq \const \|V\|_{1}$.

  Simple first order perturbation theory implies that for small
  $\lambda$, the operator (\ref{thmG}) has negative eigenvalues
  $\lambda f(e)( a^i_S + O(\lambda))$. Moreover, the $O(\lambda)$ term
  depends continuously on $e$. Thus, for every $a^i_S<0$ and
  $\lambda>0$ small enough, there exists an $e_i(\lambda)$ such that
  $\lambda f(e_i(\lambda))( a^i_S + O(\lambda))= -1$. This implies the
  statement.

  A similar argument can be applied in the case $r\geq 2$. Although
  $M_e$ is not uniformly bounded in this case, we see that for values
  of $\lambda$ and $e$ such that $\lambda f(e)$ is bounded, $\lambda
  g(e)$ goes to zero as $\lambda$ and $e$ go to zero. Because of
  Lemma~\ref{lem1} this implies that $\lambda \|M_e\|\to 0$ as
  $\lambda\to 0$ for such $e$.  Hence we can again find a function
  $e_i(\lambda)$, with $\lim_{\lambda\to 0} \lambda f(e_i(\lambda)) =
  -1/a^i_S$, such that (\ref{thmG}) has an eigenvalue $-1$ and for
  $e=e_i(\lambda)$. This concludes the proof of (i) in the general
  case.

  In order to prove (ii) we shall again apply simple perturbation
  theory, which implies that for $e=e_i(\lambda)$ the eigenvector
  $u^i_\lambda \in L^2 (S)$ of (\ref{thmG}) corresponding to the
  eigenvalue $-1$ satisfies
  \begin{equation}\label{ulambda}
    u^i_\lambda = u_i + \eta_\lambda, \qquad \lim_{\lambda \to 0} \| \eta_\lambda \|_{2}
    = 0\,,
  \end{equation}
  with $u_i$ in the eigenspace of $\V_S$ corresponding to the
  eigenvalue $a_S^i$. Applying the second part of
  Proposition~\ref{prop1}, the eigenvector of $H_\lambda$
  corresponding to the eigenvalue $-e_i(\lambda)$ equals
  \begin{equation}
    \psi_\lambda^i = \frac 1{T+e_i(\lambda)} |V|^{1/2} \frac 1{1+\lambda M_{e_i(\lambda)}} 
    V^{1/2} F_S^* \left( u_i + \eta_\lambda\right)\,.
  \end{equation}
  Using the eigenvalue equation for $\psi_\lambda$,
  $(T+e_i(\lambda))\psi_\lambda^i = -\lambda V \psi_\lambda^i$, this
  can be rewritten as
  \begin{equation}
    -\lambda V^{1/2} \psi_\lambda^i = \frac{1}{1+\lambda M_{e_i(\lambda)}} V^{1/2} 
    F_S^* \left( u_i + \eta_\lambda \right)\,.
  \end{equation}
  Now $\lambda \|M_{e_i(\lambda)}\|\to 0$ as $\lambda\to 0$, and $F_S
  |V| F_S^*$ is bounded. After appropriate normalization,
  $V^{1/2}\psi_\lambda^i$ therefore converges to $V^{1/2} F_S^* u_i$
  strongly in $L^2(\R^n)$, as claimed.

  A simple perturbation argument leads to (iii). In fact, any negative
  eigenvalue of (\ref{thmG}) which does not correspond to a negative
  eigenvalue of $\V_S$ for $\lambda=0$ can be at most as negative as
  $-\lambda f(e) \|F_S G(\lambda,e) F_S^*\|\geq - \const \lambda^2
  f(e)$ for some constant depending only on $V$. This can be easily
  seen using (\ref{gb}) and Lemma~\ref{lem1}. Hence $\lambda^2
  f(e)\geq \const$ for such eigenvalues.

  To see (iv) we use the operator inequality $G(\lambda, e) \geq -
  \const \lambda |V|$ for small $\lambda$, which follows easily from
  (\ref{Gl}) and Lemma~\ref{lem1}. The operator in (\ref{thmG}) is
  therefore bounded from below by
  \begin{equation}
    \lambda f(e) \left( \V_S - \const \lambda F_S|V| F_S^*\right) 
    = \lambda f(e) F_S\left(V - \const \lambda |V|\right)F^*_S\, ,
  \end{equation}
  which is non-negative for $\lambda$ small enough according to our
  assumption.
\end{proof}

\begin{proof}[Proof of Theorem \ref{thm2statement}]
  Since $r<2$ by assumption, Lemma \ref{lem2} implies that $M_e$
  converges to $M_0$ in operator norm. Since $V^{1/2} F_S^*$ is a 
  bounded operator, we conclude that also $F_S |V|^{1/2} M_e V^{1/2}
  F_S^*$ converges in operator norm to $F_S |V|^{1/2} M_0 V^{1/2}
  F_S^*$, which we shall denote by $\W_S$ as in (\ref{defws}).

  With $\B_S = \V_S - \lambda W_S$ as in (\ref{defB}), the operator
  \eqref{volleoperatorvonSnachS} can thus be rewritten as
  \begin{equation}\label{xxv}
    \lambda f(e)( \B_S + \lambda F_S|V|^{1/2} W(\lambda,e) V^{1/2} F^*_S)\,,
  \end{equation}
  where
  \begin{equation}
    W(\lambda,e) =   \frac{\lambda M_e^2}{1+\lambda M_e} - M_e + M_0 
  \end{equation}
  has the property that $\|W(\lambda,e)\| \to 0$ as $\lambda\to 0$ and
  $e\to 0$.  If $b_S^i(\lambda)$ is a negative eigenvalue of $\B_S$,
  with $\lim_{\lambda \to 0} b_S^i(\lambda) <0$, then a similar
  perturbation argument as in the proof of  Theorem~\ref{Thm1gen} implies
  that $ \B_S + \lambda F_SW(\lambda,e) F^*_S$ has an eigenvalue with
  the asymptotic behavior $ b_S^i(\lambda) + o(\lambda )+ \lambda o(1)$, the
  last term going to zero as $e\to 0$. Given such a $b_S^i(\lambda)$,
  we can thus find an $e_i(\lambda)$, going to zero as $\lambda\to 0$,
  such that (\ref{xxv}) has an eigenvalue $1$ for $e=e_i(\lambda)$. In the limit
  $\lambda\to 0$, we conclude that
  \begin{equation}
    \lambda f(e_i(\lambda))= - 1/(b_S^i(\lambda) + o(\lambda))\,.
  \end{equation}
  Using again Proposition \ref{prop1} we obtain (i).

  If $\V_S$ has $0$ as an eigenvalue, with corresponding eigenvector
  $u_0$, then by the definition (\ref{defB}) of $\B_S$ and the fact
  that $\V_S u_0 = F_SVF_S^*u_0 = 0$ we obtain that
  \begin{equation}
    (u_0, \B_S u_0) = -\lambda (u_0, \W_S u_0) = -
    \lambda \lim_{e \to 0} (u_0,F_S V \frac 1{T+e} V F^*_S u_0) \,.
  \end{equation}
  The latter quantity is strictly negative, as can be seen by an
  analyticity argument similar to the proof of Corollary~\ref{Cor}.
  In particular, if the kernel of $\V_S$ is not empty then there is at
  least one corresponding negative eigenvalue of $\B_S$ for small
  enough $\lambda$ and $e$.  Together with Proposition~\ref{prop1}
  this implies the existence of a corresponding negative eigenvalue of
  $H_\lambda$.
\end{proof}

We are left with proving Lemmas~\ref{lem1} and~\ref{lem2}.

\begin{proof}[Proof of Lemma~\ref{lem1}]
  We note that $\tilde M_e = \sgn(V) M_e$ is selfadjoint, and $\|M_e\|
  =\|\tilde M_e\|$. For $\psi \in L^2(\R^n)$, let $\varphi = |V|^{1/2}
  \psi$. By the definition of $M_e$ in (\ref{defMe}), we have
  \begin{equation}\label{menew}
    (\psi, \tilde M_e \psi)  = \int_{\R^n} \frac  {|\hat \varphi(p)|^2} {T(p) + e}dp  
    - {f(e)}\int_S
    \frac{|\hat \varphi(p)|^2}{|\nabla P(p)|} dp \,.
  \end{equation}
  By our assumptions on $T$, there exists a $\tau>0$ such that
  \begin{equation}
    \Omega_\tau = \{ p \in \R^3\, |\, |P(p)| < \tau \}
  \end{equation}
  is a subset of $\Omega$. Recall that $P$ is assumed to be twice
  differentiable on $\Omega$, and hence also on $\Omega_\tau$. If $S$
  is not connected, we choose $\tau$ small enough such that
  $\Omega_\tau$ has the same number of connected components as $S$. On
  $\Omega_\tau$, we will use the co-area formula to split the volume
  integral in the first term on the right side of (\ref{menew}) into
  integrals over the level sets
  \begin{equation}
    S_t =\{p \in \Omega_\tau \,|\, |P(p)| = T^{1/r}(p) = t\}
  \end{equation}
  for $0\leq t\leq \tau$. Note that $S_0=S$. In fact, using the co-area formula we have
  \begin{equation}\label{obound}
    \int_{\Omega_\tau}  \frac  {|\hat \varphi(p)|^2}  {T(p) + e}dp = \int_0^{\tau}  dt
    \frac 1{t^{r} + e} \int_{S_t} \frac{|\hat \varphi (p)|^2}{|\nabla P(p)|}dp\,,
  \end{equation}
  where $dp$ in the latter integral denotes the Lebesgue measure on
  $S_t$.

  Recall that $T(p)=|P(p)|^r=0$ on $S$, and $|\nabla P|\neq 0$ on
  $\Omega_\tau$.  Hence every connected component of  $S_t$ consists of two disjoint surfaces, 
one lying outside $S$ and one lying inside $S$. In
  order to bound \eqref{obound} we make use of the following lemma.

  \begin{Lemma}\label{Lemgauss}
    Let $h : \Omega_\tau \to \R$, with $h \in C^1(\Omega_\tau)$, and
    let $0<t<\tau$. Then
    \begin{equation}
      \left | \int_{S_t} h (p) dp  - 2 \int_S h(p) dp \right| \leq   \int_0^t d\sigma 
\int_{S_\sigma} \frac 1{|\nabla P(p)|} \left| \nabla \cdot 
        \left( h(p) \frac{\nabla P(p)}{|\nabla P(p)|}\right) \right| dp\,.
    \end{equation}
  \end{Lemma}

\begin{proof}
Without loss of generality we can assume that $S$ is connected. We shall write  $S_t=S^o_t
  \cup S^i_t$, with $S_t^{i,o}$ lying inside and outside $S$, respectively.
  Let $\Omega_t^{o,i} = \bigcup_{0\leq \sigma\leq t} S_t^{o,i}$ denote
  the union of the sets $S_\sigma^{o,i}$ for $0\leq \sigma\leq t$.
  By definition $\frac{\nabla P}{|\nabla P|}$ is a unit vector field
  which is orthogonal to the hypersurfaces $S^o_t$ and $S^i_t$ and
  points either inward or outward, depending on $P$. Depending on the
  direction, we have
  \begin{align}
    \int_{S^o_t} h(p) dp - \int_{S}h(p) dp =& \pm \int_{\partial
      \Omega^o_t} h(p)
    \frac{\nabla P(p)}{|\nabla P(p)|} \cdot d {\bf S} \\
    \int_{S^i_t} h(p) dp - \int_{S}h(p) dp =& \mp \int_{\partial
      \Omega^i_t} h(p) \frac{\nabla P(p)}{|\nabla P(p)|} \cdot d {\bf
      S} \,.
  \end{align}
  Using Gauss' theorem we infer, for $q = o,i$,
  \begin{multline}
    \int_{\partial \Omega^q_t} h(p) \frac{\nabla P(p)}{|\nabla P(p)|}
    \cdot d {\bf S} = \pm \int_{\Omega^q_t} \nabla \cdot \left(h(p)
      \frac{\nabla P(p)}{|\nabla P(p)|} \right)dp \\ = \pm \int_0^t
    d\tau \int_{S^q_\tau} \frac 1{|\nabla P(p)|} \nabla \cdot \left(
      h(p) \frac{\nabla P(p)}{|\nabla P(p)|}\right) dp\,,
  \end{multline}
  where the last equation follows again from the co-area formula. The
  rest is obvious.
\end{proof}

We shall now apply Lemma \ref{Lemgauss} to the function $h(p) =
|\hat\varphi(p)|^2 |\nabla P(p)|^{-1}$.  Note that
\begin{align}\nonumber
  |\hat\varphi(p)|^2 & = (2\pi)^{-n} \iint |V(x)|^{1/2} |V(y)|^{1/2}
  \psi(x)^* \psi(y) e^{ip\cdot (x-y)} dx dy \\ & \leq (2\pi)^{-n} \|V\|_1
  \|\psi\|_2^2 \label{bp}
\end{align}
uniformly in $p$ by Schwarz's inequality. Similarly
\begin{equation}\label{nabla}
  \left|\nabla|\hat\varphi(p)|^2\right| \leq (2\pi)^{-n}  \|\psi\|_2^2 
  \left( \iint |V(x)| |V(y)| |x-y|^2 dx dy \right)^{1/2} \,.
\end{equation}
By assumption, there are constants $c,C>0$ such that $|\nabla P| \geq
c$ and $|\partial_i \partial_j P| \leq C$ for $1\leq i,j\leq n$ on $\Omega_\tau$. Moreover, the measure of
the sets $S_t$ is uniformly bounded for $0\leq t\leq \tau$. We
conclude that
\begin{multline}\label{comb1}
  \left| \int_{S_t} \frac{ |\hat\varphi(p)|^2}{|\nabla P(p)|} dp - 2
    \int_S \frac {|\hat\varphi(p)|^2}{|\nabla P(p)|} dp \right|\\ \leq
  \const t \|\psi\|_2^2 \left[ \|V\|_1 +\left( \iint |V(x)| |V(y)|
      |x-y|^2 dx dy \right)^{1/2} \right] \,.
\end{multline}
By combining (\ref{comb1}) with (\ref{obound}) and (\ref{menew}), we
obtain the bound
\begin{align}\nonumber
  \left| (\psi, \tilde M_e \psi) \right| & \leq \const \int_0^\tau
  \frac{t\, dt }{t^r + e} \|\psi\|_2^2 \left[ \|V\|_1 +\left( \iint |V(x)|
      |V(y)| |x-y|^2 dx dy \right)^{1/2} \right] \\ \label{combi} &
  \quad + \left| \int_0^\tau \frac{2\, dt }{t^r + e} - f(e) \right| \int_S
  \frac{ |\hat\varphi(p)|^2}{|\nabla P(p)|} dp + \int_{\Omega_\tau^c}
  \frac{|\hat\varphi(p)|^2} { T(p) + e} dp \,.
\end{align}
It is easy to see that $\int_0^\tau t(t^2+e)^{-1} dt \leq \const g(e)$ for any
fixed $\tau$. Similarly, $|f(e) - 2\int_0^\tau (t^r+e)^{-1} dt |\leq
\const g(e)$. The integral  in the second term in (\ref{combi}) is bounded
by $\|\psi\|_2^2 \|V\|_1$ using (\ref{bp}).  Moreover, since $T(p)
\geq \const (1+|p|^s)$ on $\Omega_\tau^c$, the last term in
(\ref{combi}) can bounded with the aid of the Hardy-Littlewood-Sobolev
inequality \cite[Theorem.~4.3]{LL} and H\"older's inequality as
\begin{equation}
  \int_{\Omega_\tau^c} \frac{|\hat\varphi(p)|^2} { T(p) + e} dp \leq
  \const \|\psi\|_2^2 \mathcal{A}(V)
\end{equation}
with $\mathcal{A}(V)$ defined in (\ref{defcala}).  This proves
(\ref{menonr}) with $\kappa=2$.

In the case when $T$ is radial, the surfaces $S_t$ are $n-1$
dimensional spheres.  In this case, we can obtain a better bound on
$\|M_e\|$ in the following way. It is not necessary to obtain a
pointwise bound on $\nabla |\hat\varphi(p)|^2$ but only on its
spherical average.  Using the fact that for $n\geq 2$ 
\begin{align}\nonumber
  \frac 1{|\Ss^{n-1}|} \int_{\Ss^{n-1}} e^{i k\cdot \omega} d\omega &
  = \frac{|\Ss^{n-2}|}{|\Ss^{n-1}|} \int_0^\pi e^{i |k| \cos\theta}
  (\sin \theta)^{n-2} d\theta \\ & = \pi
  \frac{\Gamma((n-1)/2)^2}{\Gamma(n/2)} \left(\frac
    2{|k|}\right)^{(n-2)/2} J_{(n-2)/2} (|k|) \,,
\end{align}
where $J_{(n-2)/2}$ is a Bessel function, as well as
the bounds $J_{(n-2)/2}(|k|)\leq 1$, $J_{(n-2)/2}(|k|) \leq
(|k|/2)^{(n-2)/2}\Gamma((n-1)/2)$ and the asymptotics 
$J_{(n-2)/2} (|k|) \sim |k|^{-1/2}$ for $|k|\to \infty$ \cite{abra},
it is easy to see that
\begin{equation}
  \left|  \int_{\Ss^{n-1}} \nabla |\hat\varphi(|p|\omega)|^2 d\omega \right| 
  \leq \const \|\psi\|_2^2  
  \left( \iint |V(x)| |V(y)| |x-y|^\kappa dx dy\right)^{1/2} 
\end{equation}
with $\kappa=0$ for $n\geq 3$ and $\kappa=1$ for $n=2$.  Using this
bound instead of (\ref{nabla}) and proceeding as above, we arrive at
(\ref{menonr}) with $\kappa$ as stated.
\end{proof}

\begin{proof}[Proof of Lemma~\ref{lem2}]
  Let $\tilde M_0$ be defined via the quadratic form
  \begin{align}\nonumber
    (\psi, \tilde M_0 \psi) & = \int_0^\tau \frac 1{t^r} \left(
      \int_{S_t} \frac{|\hat\varphi(p)|^2}{|\nabla P(p)|} dp - 2 \int_S
      \frac{|\hat \varphi(p)|^2}{|\nabla P(p)|} dp \right) +
    \int_{\Omega_\tau^c} \frac{|\hat\varphi(p)|^2}{T(p)} dp \\ & \quad
    + C_\tau \int_S \frac{|\hat \varphi(p)|^2}{|\nabla P(p)|} dp\,,
  \end{align}
  where $\varphi = |V|^{1/2} \psi$ and
  \begin{equation}
    C_\tau = \lim_{e\to 0} C_\tau(e) \,, \quad C_\tau(e) =  
\int_0^\tau \frac{2}{t^r + e} dt - f(e) \,,
  \end{equation}
  which is finite for $1\leq r <2$.  The notation is the same as in
  the proof of Lemma~\ref{lem1}. With $\tilde M_e = \sgn(V) M_e$ as
  before, we have
  \begin{align}\nonumber
    \left( \psi, \left( \tilde M_e - \tilde M_0\right) \psi\right) & =
    \int_0^\tau dt \left( \frac{1}{t^r+e} - \frac 1{t^r}\right) \left(
      \int_{S_t} \frac{|\hat\varphi(p)|^2}{|\nabla P(p)|} dp 
  -2  \int_{S} \frac{|\hat\varphi(p)|^2}{|\nabla P(p)|} dp\right)  \\
    \nonumber & \quad + \int_{\Omega_\tau^c} |\hat\varphi(p)|^2 \left(
      \frac{1}{T(p)+e}- \frac{1}{T(p)}\right) dp \\ & \quad + \left(
      C_\tau(e) - C_\tau\right) \int_S \frac{|\hat
      \varphi(p)|^2}{|\nabla P(p)|} dp \,.
  \end{align}
  From this representation and the various bounds derived in the proof
  of Lemma~\ref{lem1}, it is easy to see that the right side goes to
  zero as $e\to 0$, and the convergence is uniform in $\psi$ for fixed
  $\|\psi\|_2$. This implies that $\lim_{e\to 0} \|\tilde M_e - \tilde
  M_0\| =0$, and hence also $\lim_{e\to 0} \|M_e - M_0\| =0$ with
  $M_0= \sgn(V) \tilde M_0$.
\end{proof}

\bigskip \noindent {\it Acknowledgements.}  Part of this work was done during the
authors' visit at the Erwin Schr\"odinger Institute for Mathematical Physics in
Vienna, Austria, and the hospitality and support during this visit is gratefully
acknowledged. This work was partially supported by U.S. National Science
Foundation grants DMS-0800906 (C.H.) and PHY-0652356 (R.S.).


\end{document}